\documentclass{article}
\usepackage{amsmath,amscd,amsfonts,amssymb,amsthm}
\usepackage[latin1]{inputenc}
\usepackage{color,a4wide}

\newcommand{\subj}[1]{\par\noindent{\bf Mathematics Subject Classification 2010: }#1.}
\newcommand{\keyw}[1]{\par\noindent{\bf Keywords: }#1.}
\theoremstyle{definition}
\newtheorem{definition}{Definition}
\newtheorem{theorem}{Theorem}

\def\a{\alpha}

\def\f{f^{(\a)}}

\def\DS{\displaystyle}

\begin{document}

\title{A remark on local fractional calculus}

\author{Ricardo Almeida$^1$\\
{\tt ricardo.almeida@ua.pt}
\and
Ma\l gorzata Guzowska$^{2}$\\
{\tt mguzowska@wneiz.pl}
\and
Tatiana Odzijewicz$^{3}$\\
{\tt tatiana.odzijewicz@sgh.waw.pl}
}

\date{$^1$Center for Research and Development in Mathematics and Applications\\
Department of Mathematics, University of Aveiro, 3810--193 Aveiro, Portugal\\
$^2$ Faculty of Economics and Management\\ University of Szczecin,
 71-101 Szczecin, Poland\\
$^3$Department of Mathematics and Mathematical Economics\\
Warsaw School of Economics,
 02-554 Warsaw, Poland}

\maketitle

% ---------------------------------

\begin{abstract}

In this short note  we present a new general definition of local fractional derivative, that depends on an unknown kernel. For some appropriate choices of the kernel we obtain some known cases. We obtain a relation between this new concept and ordinary differentiation. Using such formula, most of the fundamental properties of the fractional derivative can be derived directly.

\end{abstract}

\subj{local fractional derivative, conformable derivative}

\keyw{26A33,26A24}

% ---------------------------------

\section{Introduction}

Fractional calculus is a generalization of ordinary calculus, where derivatives and integrals of arbitrary real or complex order are defined. These fractional operators may model more efficiently certain real world phenomena, specially when the dynamics is affected by constraints inherent to the system. There exist several definitions for fractional derivatives and fractional integrals, like the Riemann--Liouville, Caputo, Hadamard, Riesz, Gr\"{u}nwald--Letnikov, Marchaud, etc. (see e.g., \cite{Kilbas,Podlubny} and references therein). Although most of them are already well-studied, some of the usual features for what concerns the differentiation of functions fails, like the Leibniz rule, the chain rule, the semigroup property, among others. As it was mentioned in \cite{Babakhani}, ``These definitions, however, are non-local in nature, which makes them unsuitable for investigating properties related to local scaling or fractional differentiability". Recently, the concept of local fractional derivative have gained relevance, namely because they kept some of the properties of ordinary derivatives, although they loss the memory condition inherent to the usual fractional derivatives \cite{Anderson,Chen,Katumgapola,Kolwankar,Kolwankar2}.
 One  question is what is the best local fractional derivative definition that we should consider, and the answer is not unique. Similarly to what happens to the classical definitions of fractional operators, the best choice depends on the experimental data that fits better in the theoretical model, and because of this we find already a vast number of definitions for local fractional derivatives.

\section{Local fractional derivative}

We present a definition of local fractional derivative using kernels.

\begin{definition} Let $k:[a,b]\to\mathbb R$ be a continuous nonnegative map  such that $k(t)\not=0$, whenever $t>a$. Given a function $f:[a,b]\to\mathbb R$ and $\a\in(0,1)$ a real, we say that $f$ is $\a$-differentiable at $t>a$, with respect to kernel $k$, if the limit
\begin{equation}\label{def}\f(t):=\lim_{\epsilon\to0}\frac{f(t+\epsilon k(t)^{1-\a})-f(t)}{\epsilon}\end{equation}
exists. The $\a$-derivative at $t=a$ is defined by
$$\f(a):=\lim_{t\to a^+}\f(t),$$
if the limit exists.
\end{definition}
Consider the limit $\a\to1^-$. In this case, for $t>a$, we obtain the classical definition for derivative of a function, $\f(t)=f'(t)$. Our definition is a more general concept, compared to others that we find in the literature. For example, taking $k(t)=t$ and $a=0$, we get the definition from \cite{Batarfi,Cenesiz,Hammad,Hesameddini,Khalil} (also called conformable fractional derivative); when $k(t)=t-a$, the one from \cite{Abdeljawad,Anderson2,Unal}; for $k(t)=t+1/\Gamma(\a)$, the definition in \cite{Atangana,Atangana2}.

The following result is trivial, and we omit the proof.

\begin{theorem} Let  $f:[a,b]\to\mathbb R$ be a differentiable function and $t>a$. Then, $f$ is $\alpha$-differentiable at $t$ and
$$\f(t)=k(t)^{1-\a}f'(t), \quad t>a.$$
Also, if $f'$ is continuous at $t=a$, then
$$\f(a)=k(a)^{1-\a}f'(a).$$
\end{theorem}

However, there exist $\a$-differentiable functions which are not differentiable in the usual sense. For example, consider the function $f(t)=\sqrt t$, with $t\geq0$. If we take the kernel $k(t)=t$, then $\f(t)=1/2 \, t^{1/2-\a}$. Thus, for $\a\in(0,1/2)$, $\f(0)=0$ and for $\a=1/2$, $\f(0)=1/2$. In general, if we consider the function  $f(t)=\sqrt[n]{t}$, with $t\geq0$ and $n\in\mathbb N\setminus \{1\}$, we have $\f(t)=1/n \, t^{1/n-\a}$ and so $\f(0)=0$ if $\a\in(0,1/n)$ and for $\a=1/n$, $\f(0)=1/n$.

\begin{theorem}  If $\f(t)$ exists for $t>a$, then $f$ is differentiable at $t$ and
$$f'(t)=k(t)^{\a-1} \f(t).$$
\end{theorem}
\begin{proof} It follows from
$$\begin{array}{ll}
f'(t)&=\DS \lim_{\delta\to0}\frac{f(t+\delta)-f(t)}{\delta}\\
&=\DS k(t)^{\a-1} \lim_{\epsilon\to0}\frac{f(t+\epsilon k(t)^{1-\a})-f(t)}{\epsilon}\\
&=\DS k(t)^{\a-1} \f(t).
\end{array}$$
\end{proof}

Of course we can not conclude anything at the initial point $t=a$, as was discussed before.

Combining the two previous results, we have the main result of our paper.

\begin{theorem}\label{MainT} A function  $f:[a,b]\to\mathbb R$ is $\a$-differentiable at $t>a$ if and only if it is differentiable at $t$. In that case, we have the relation
\begin{equation}\label{MainF}\f(t)=k(t)^{1-\a}f'(t), \quad t>a.\end{equation}
\end{theorem}

\section{Conclusion}

In this short note we show that some of the existent notions about local fractional derivative are very close related to the usual derivative function. In fact, the $\alpha$-derivative of a function is equal to the first-order derivative, multiplied by a continuous function. Also, using formula \eqref{MainF}, most of the results concerning $\alpha$-differentiation can be deduced trivially from the ordinary ones. In the authors opinion, local fractional calculus is an interesting idea and deserves further research, but definitions like \eqref{def} are not the best ones and a different path should be followed.

%% ===================================================

\section*{Acknowledgments}

Research supported by Portuguese funds through the CIDMA - Center for Research and Development in Mathematics and Applications,
and the Portuguese Foundation for Science and Technology (FCT-Funda\c{c}\~ao para a Ci\^encia e a Tecnologia), within project UID/MAT/04106/2013 (R. Almeida) and by the Warsaw School of Economics grant KAE/S15/35/15 (T. Odzijewicz).

%% ===================================================


\begin{thebibliography}{99}

\bibitem{Abdeljawad}
T. Abdeljawad,
On conformable fractional calulus, preprint.

\bibitem{Anderson2}
D.R. Anderson and R.I. Avery,
Fractional-order boundary value problem with Sturm--Liouville boundary conditions, Electron. J. Differ. Equ., Volume 2015, 29,  1--10, 2015.

\bibitem{Anderson}
D.R. Anderson and D.J. Ulness,
Properties of the Katugampola fractional derivative with potential application in quantum mechanics, J. Math. Phys 56, 063502, 2015.

\bibitem{Atangana}
A. Atangana and E.F.D. Goufo,
Extension of Matched Asymptotic Method to Fractional Boundary Layers Problems, Mathematical Problems in Engineering, Volume 2014, 107535, 7 pages.

\bibitem{Atangana2}
A. Atangana and S.C.O. Noutchie,
Model of Break-Bone Fever via Beta-Derivatives, BioMed Research International, Volume 2014, 523159, 10 pages.

\bibitem{Babakhani}
A. Babakhani and V. Daftardar--Gejji,
On calculus of local fractional derivatives, J. Math. Anal. Appl. 270 (1), 66--79,  2002.

\bibitem{Batarfi}
H. Batarfi, J. Losada, J.J. Nieto and W. Shammakh,
Three-point boundary value problems for conformable fractional differential equations, Journal of function spaces, Volume 2015, 706383, 6 pages.

\bibitem{Cenesiz}
Y. \c{C}enesiz and A. Kurt,
The solution of time fractional heat equation with new fractional derivative definition, in Recent Advances in Applied Mathematics, Modelling and Simulation (eds N.E. Mastorakis, M. Demiralp, N. Mukhopadhyay and F. Mainardi) North Atlantic University Union, 195--198, 2014.

\bibitem{Chen}
Y. Chen, Y. Yan  and K. Zhang,
On the local fractional derivative,  J. Math. Anal. Appl. 362 (1), 17--33,  2010.

\bibitem{Hammad}
M. Abu Hammad and R. Khalil,
Legendre fractional differential equation and Legender fractional polynomials, Int. J. Appl. Math. Res. 3 (3), 214--219, 2014.

\bibitem{Hesameddini}
E. Hesameddini and E. Asadollahifard,
Numerical solution of multi-order fractional differential equations via the sinc collocation method, Iran. J. Numer. Anal. Optim. 5 (1), 37--48, 2015.

\bibitem{Katumgapola}
U. Katumgapola,
A new fractional derivative with classical properties, preprint.

\bibitem{Khalil}
R. Khalil, M. Al Horani, A. Yousef and M. Sababheh,
A new definition of fractional derivative, J. Comput. Appl. Math. 264. 65--70, 2014.

\bibitem{Kilbas}
A.A. Kilbas, H.M. Srivastava and J.J. Trujillo, Theory and Applications of Fractional Differential Equations.
North-Holland Mathematics Studies, 204. Elsevier Science B.V., Amsterdam, 2006.

\bibitem{Kolwankar}
K.M. Kolwankar and A.D. Gangal,
Fractional differentiability of nowhere differentiable functions and dimension, Chaos 6, 505--513, 1996.

\bibitem{Kolwankar2}
K.M. Kolwankar and A.D. Gangal,
H\"{o}lder exponents of irregular signals and local fractional derivatives, Pramana J. Phys. 48, 49--68, 1997.

\bibitem{Podlubny}
I. Podlubny,
Fractional differential equations, Mathematics in Science and Engineering, 198. Academic Press, Inc., San Diego, CA, 1999.

\bibitem{Unal}
E. \"{U}nal, A. G\v{o}kdogan and E. \c{C}elik,
Solutions around a regular $\a$ singular point of a sequential conformable fractional differential equation, preprint.


\end{thebibliography}
\end{document}